\documentclass[11pt]{amsart}

\usepackage{amsmath,amssymb,amsthm}
\usepackage[english]{babel}
\usepackage[colorlinks=true,linkcolor=blue,citecolor=blue,pdfpagelabels=false]{hyperref}
\usepackage{tikz-cd}
\usepackage{xfrac} 
\usepackage{faktor}

 2

\newtheorem{thm}{Theorem}[section]
\newtheorem{lem}[thm]{Lemma}

\newtheorem{cor}[thm]{Corollary}
\newtheorem{defn}[thm]{Definition}
\newtheorem{rmk}[thm]{Remark}

\newcommand{\thmref}[1]{Theorem~\ref{#1}}
\newcommand{\lemref}[1]{Lemma~\ref{#1}}
\newcommand{\rmkref}[1]{Remark~\ref{#1}}

\newcommand{\Z}{\mathbb{Z}}
\newcommand{\Q}{\mathbb{Q}}
\newcommand{\R}{\mathbb{R}}

\newenvironment{acknowledgements}{\bigskip\textbf{Acknowledgements.}}{}

\makeatletter
\@namedef{subjclassname@2020}{%
\textup{2020} Mathematics Subject Classification}
\makeatother

\begin{document}

	\title{On Signs of Hecke Eigenvalues of Siegel eigenforms}
	
	\author{Arvind Kumar, Jaban Meher and Karam Deo Shankhadhar}
	
\address[Arvind Kumar]
%{Einstein Institute of Mathematics, The Hebrew University of Jerusalem, Edmund Safra Campus, Jerusalem 91904, Israel. %\newline {\em Current address: } 
	{{Department of Mathematics, Indian Institute of Technology Jammu, Jagti NH-44, PO Nagrota, Jammu 181221, India.}}
%\email{arvind.kumar@mail.huji.ac.il}
% {Department of Mathematics, Indian Institute of Technology Jammu, Jagti NH-44, PO Nagrota, Jammu 181221, India.}}
\email{arvind.kumar@iitjammu.ac.in}
	
	\address[Jaban Meher]{School of Mathematical Sciences, 
		National Institute of Science Education and Research, Bhubaneswar, An OCC of Homi Bhabha National Institute, P. O. Jatni, Khurda 752050, Odisha, India.}
	\email{jaban@niser.ac.in}
	
	\address[Karam Deo Shankhadhar]{Department of Mathematics,
		Indian Institute of Science Education and Research
		Bhopal, Bhopal Bypass Road, Bhauri, Bhopal 462 066, Madhya Pradesh, India.}
	\email{karamdeo@iiserb.ac.in}
	
	
	%\date{\today}
	
	
	\begin{abstract}
		In this article, we distinguish Siegel cuspidal eigenforms of degree two on the full symplectic group from the signs of their Hecke eigenvalues. To establish our theorem, we obtain a result towards simultaneous sign changes of eigenvalues of two Siegel eigenforms. In the course of the proof, we also prove that the Satake $p$-parameters of two different Siegel eigenforms are distinct for a set of primes $p$ of density 1. The main ingredient to prove the latter result is the theory of Galois representations attached to Siegel eigenforms. %These intermediate results are also of independent interest.  
	\end{abstract}
	
	\maketitle
	\section{Introduction}
	 One of the fundamental and interesting problems in the theory of automorphic forms is {\em whether the set of 
	 { Hecke} eigenvalues  determine the eigenform under consideration?} This question is well studied for elliptic modular eigenforms and there are many results available in the literature. But in the case of Siegel cuspidal eigenforms, this was a long-standing unanswered problem and only recently in 2018, Schmidt \cite{sch} has given an affirmative answer for normalized eigenvalues. 
	To make our statements more concrete, let us introduce some notation. Let $S_k(\Gamma_2)$ be the space of Siegel cusp forms of weight $k$ for the symplectic group $\Gamma_2:= {\rm Sp}_4(\Z)$. 
	For an eigenform  $F\in S_k(\Gamma_2)$, we denote the $n$-th eigenvalue by $\mu_F(n)$ and the $n$-th normalized eigenvalue by $\lambda_F(n):=n^{3/2-k}\mu_F(n)$.
	Then Schmidt {\em loc. cit.} has proved that if $F\in S_{k_1}(\Gamma_2)$ and $G\in S_{k_2}(\Gamma_2)$ 
	are eigenforms %with respective $n$-th normalized eigenvalues $\lambda_F(n)$ and  $\lambda_G(n)$ 
	such that for all but finitely many primes $p$
	$$
	\lambda_F(p)=\lambda_G(p) ~~~{\rm and}~~~	\lambda_F(p^2)=\lambda_G(p^2),
	$$
	then $k_1 = k_2$ and $F$ is a constant multiple of $G$. In the literature, this kind of result is known as the multiplicity one theorem. The result of Schmidt has been improved significantly by the authors of this paper in \cite{kms} 
	by proving a strong multiplicity one result which essentially shows that any set	of eigenvalues (normalized or non-normalized) at primes $p$ of positive upper density are sufficient
	to distinguish the Siegel cuspidal eigenform. 
	
%	$$
%	\limsup_{x \rightarrow \infty} \frac{\#\{p \le x %~{\rm and}~p~{\rm is~a~prime}		:	 \lambda_F(p)= 	\lambda_G(p)\}}{x/\log x} > 0 \quad	{\rm or} \quad	\limsup_{x \rightarrow \infty} \frac{\#\{p \le x :	 \mu_F(p)= 	\mu_G(p)\}}{x/\log x} > 0,	$$
%	then $k_1 = k_2$ and	$F$ is a constant multiple of $G$.  In other words, any set of eigenvalues	at primes $p$ of positive upper density is sufficient to distinguish a 	Siegel cuspidal eigenform.

%	Let us denote the subspace of non-Saito-Kurokawa lifts inside  $S_k(\Gamma_2)$ by 	 $S_k^{\perp}(\Gamma_2)$.Recently, Chen and Li \cite[Theorem 1.2]{chen} have studied the above problem for eigenvalues indexed by positive integers. By using the theory of $\mathfrak B$-free numbers, they have proved that if  $F\in S_{k_1}^\perp(\Gamma_2)$ and $G\in S_{k_2}^\perp(\Gamma_2)$ are Siegel eigenforms, not a constant multiple of each other, then for sufficiently large $x$ one has
%	$$	\#\{n\le x: \lambda_F(n)\neq \lambda_G(n) \} \gg x.$$
 %This result of Chen et al. and the proof given can easily seen to be refined for eigenvalues indexed by squarefree positive integers by making use of the above mentioned result of \cite{kms}. A similar result can be deduced for non-normalized eigenvalues as well.
	 
	Here is another and more refined question one can ask: 
	\begin{equation}\label{question}
		\text{\em To what extent the signs of the eigenvalues determine
			an automorphic form uniquely?}
	\end{equation} 
For elliptic modular forms, this problem 
was first studied by Kowalski {\em et al.} \cite{klsw} by using a very deep result of Ramakrishnan stating that the Rankin-Selberg convolution $L$-function associated to two elliptic newforms is the $L$-function of some  ${\rm GL}_4$-form. Soon after, 
Matom\"aki \cite{mat} refined some of the results of \cite{klsw}. Loosely speaking, she has proved that if the signs of  $p$-th Hecke eigenvalues of two non-CM elliptic newforms  agree on a set of primes $p$ of analytic density greater than $19/25$, then the two forms are the same. %The proof uses the Sato-Tate conjecture.% which has been recently proved.

In this article, we investigate the  question (\ref{question}) in the case of Siegel cuspidal eigenforms. Let us first recall some facts about the signs of their Hecke eigenvalues to put things in order. If $F$ is a Siegel eigenform with normalized eigenvalues $\lambda_F(n)$, then we know that $\lambda_F(n)$ are real numbers for all $n \ge 1.$ By the works of Breulmann \cite{bre} and Kohnen \cite{kohnen}, one knows that  $F$ is a Saito-Kurokawa lift if and only if   $\lambda_F(n)>0$ for all $n\ge 1$. Therefore one should discuss about the signs of $\lambda_F(n)$ only if $F$ is not a Saito-Kurokawa lift. %The first result about the sign change of $\lambda_F(n)$ was obtained by Kohnen \cite{kohnen} proving  that the sequence $\{\lambda_F(n)\}_{n\ge 1}$ changes sign infinitely often.
Let us denote the subspace of non-Saito-Kurokawa lifts inside $S_k(\Gamma_2)$ by $S_k^\perp(\Gamma_2)$.
 By using representation theoretic techniques and a weaker result than the generalized Ramanujan conjecture, Pitale and Schmidt \cite{pisc} have shown that for any Siegel eigenform $F\in S_k^\perp(\Gamma_2)$, there are infinitely many primes $p$ such that the sequence $\{\lambda_F(p^r)\}_{r \ge 0}$ has infinitely many sign changes.  This result has been strengthened by Das and Kohnen \cite[Theorem 1.1]{dm} by using a completely different method. They have proved that for any positive integer $j$ with $4\nmid j$, there exists a set of primes of natural density 1 such that for any prime $p$ in that set, the sequence  $\{\lambda_F(p^{jr})\}_{r\ge 0}$  has infinitely many sign changes. 
 
% To address the question (\ref{question}) 
{ Next we ask that if given any two distinct} non-Saito-Kurokawa Siegel eigenforms $F$ and $G$, there exists at least one integer $m\ge 1$ such that the sign of $\lambda_F(m)$ is different from the sign of $\lambda_G(m)$.  
{ Unlike} the case of elliptic modular forms, there is no unconditional result known till now about such simultaneous sign change.   
 %$Probably, one of the main reasons is the non-availability of Sato-Tate distribution in this case. 
 %However, in contrast to elliptic modular forms, to the best of our knowledge,  there is no unconditional result known in the literature about such simultaneous sign change of the eigenvalues. Probably, one of the main reasons is the non-availability of Sato-Tate distribution in this case. 
 %Before answering question (\ref{question}), 
We first prove the following result on the simultaneous sign change of eigenvalues of two distinct Siegel eigenforms.

	\begin{thm}\label{prime_power}
		Let $F \in S_{k_1}^\perp(\Gamma_2)$ and $G \in S_{k_2}^\perp(\Gamma_2)$ be Siegel eigenforms  with normalized eigenvalues $\lambda_F(n)$ and $\lambda_G(n)$, respectively. We assume that they are not a constant multiple of each other. Then there exists a set of primes of density $1$ such that for any prime $p$ in that set, the sequence $\{\lambda_F(p^r)\lambda_G(p^r)\}_{r\ge 0}$ has infinitely many sign changes.
	\end{thm}

	To prove the above result, \thmref{satake} plays a vital role which is the most technical part of this paper. It states that the Satake $p$-parameters of any two Siegel eigenforms  $F$ and $G$, as in \thmref{prime_power}, are distinct for a set of primes $p$ of density 1. To establish this result, we make use of certain tools from the theory of Galois representations attached to Siegel eigenforms.
		\begin{rmk}
		{\rm 
			We want to emphasize that the ideas used in the proof of \thmref{prime_power} can be adopted to prove a similar result for some other automorphic forms. For example, one can easily prove such simultaneous sign change result 
			%for prime power indices (for a density 1 set of primes) for eigenvalues 
			for eigenvalues of  two non-CM elliptic newforms such that neither of them is a Galois-conjugate to a twist of the other, which generalizes \cite[Theorem 3]{gkp}. In the elliptic case, we need to  use \cite[Theorem 3.2.2]{loe} to prove an analogous result to \thmref{satake}.}
	\end{rmk}
	
%	\begin{rmk}
%		It is important to point out that unlike to the case of elliptic modular forms there is no unconditional result known in the literature about simultaneous sign change of the eigenvalues of two Siegel eigenforms.  Probably, one of the main reasons is the non-availability of Sato-Tate distribution in this case.
%	\end{rmk}
	
%	Now we come back to our main question (\ref{question}).
%	For (elliptic) newforms, the problem of distinguishing them by the signs of their eigenvalues 
%	was first studied by Kowalski {\em et al.} \cite{klsw} by using a very deep result of Ramakrishnan stating that the Rankin–Selberg convolution $L$-function associated to two elliptic newforms is the $L$-function of some modular form on ${\rm GL}_4$. Soon after, 
%	Matom\"aki \cite{mat} refined some of the results of \cite{klsw}. Loosely speaking, she has proved that if the signs of  $p$-th Hecke eigenvalues of two non-CM elliptic newforms  agree on a set of primes $p$ of analytic density at least $19/25$, then the two forms are the same. The proof uses the Sato–Tate conjecture which has
%	been recently proved. 
 %Since such powerful machineries are not available in the case of Siegel eigenforms, the ideas used in \cite{klsw, mat} can not be applied in this case. However, 
 Now we  state our next result which answers question (\ref{question}) for Siegel modular forms. 
 %Using \thmref{prime_power} and the theory of $\mathfrak B$-free numbers, we prove the following.
%obtain a characterization of similar nature of Siegel eigenforms by the signs of eigenvalues indexed by positive integers. We now state our main result which answers the above-stated question. 
	\begin{thm}\label{thm3}
		Let $F \in S_{k_1}^\perp(\Gamma_2)$ and $G \in S_{k_2}^\perp(\Gamma_2)$ be Siegel eigenforms  with normalized eigenvalues $\lambda_F(n)$ and $\lambda_G(n)$, respectively.
		If
		$$
		\liminf_{x\rightarrow \infty }\frac{\#\{n\le x: {\rm sign} \left( \lambda_F(n) \right)\neq {\rm sign}\left(  \lambda_G(n)\right) \}}{x} =0,
		$$
		then $k_1=k_2$ and $F$ is a scalar multiple of $G$.
	\end{thm}
%	In particular, if the signs of $n$-th Hecke eigenvalues of two non-Saito-Kurokawa Siegel eigenforms agree on a set of positive integers $n$ of density 1, then the two forms are the same (up to a constant). 
	
	The main ingredient in the proof of \thmref{thm3} is the existence of a prime power $p_0^t$ for which $\lambda_F(p_0^t)\lambda_G(p_0^t)<0$, governed by \thmref{prime_power}. We want to emphasize that if one could manage to get a prime $p_0$ such that $\lambda_F(p_0) \lambda_G(p_0) < 0$ then \thmref{thm3} can be refined to distinguish $F$ and $G$ by the signs of their eigenvalues indexed by squarefree positive integers. Moreover, it would be of interest to study a similar problem for prime indices.

%	\begin{rmk} 
%		{\rm 
%			Let $F$ and $G$ be as in \thmref{prime_power}.  Note that the arguments used in the proof of \thmref{prime_power} are similar	in spirit to \cite{dm}; namely, to analyze the $p$-Euler factors of the Rankin-Selberg zeta functions attached to $F$ and $G$.Combining the ideas used in the proofs of \cite[Theorem 1.1]{dm} and \thmref{prime_power},  we believe that one should be able to  prove that for any positive integer $j$ with $4 \nmid j$, the set of primes $p$ for which the sequence $\{\lambda_F(p^{jr})\lambda_G(p^{jr})\}_{r\ge 0}$ has infinitely many sign changes has density 1.		}
	%\end{rmk}

	\subsection*{Applications}
	We now use the above results  to obtain similar results for the Fourier coefficients of a Siegel eigenform. In particular, under certain assumptions, it is possible to characterise Siegel eigenforms in terms of signs  of its Fourier coefficients supported on matrices of the form $n T_0$, where $T_0$ is a certain fixed symmetric, half-integral, positive definite $2 \times 2$ matrix and $n$ varies over a set of integers of positive lower density.
	
	Let $F \in S_k(\Gamma_2)$ be a Siegel eigenform with eigenvalues $\mu_F(n)$ and Fourier coefficients $A_F(T)$. Suppose $T_0$ is a symmetric, half-integral, 
	positive-definite matrix such that 
	$-\det(2T_0) = -D_0$ is a fundamental discriminant and $\mathbb{Q}(\sqrt{-D_0})$ 
	has class number $1$. Then \cite[Theorem 2.4.1]{and} simplifies to give the following relation:
	\begin{equation*}
		\sum_{n \ge 1} \frac{A_F(nT_0)}{n^s} = A_F(T_0) \sum_{n \ge 1} \frac{\mu_F(n)}{n^s},
	\end{equation*}
	{in some right half plane Re$(s) \gg 1$}. In particular, if $A_F(T_0) \neq 0$ then 
	
	\begin{equation}\label{eq:relation}
		\frac{A_F(nT_0)}{A_F(T_0)} =\mu_F(n),\quad  \text{for~all }\quad n\ge 1.
	\end{equation} 
	Since the sign of $\mu_F(n)$ and $\lambda_F(n)$ are the same, by using the above relation together with Theorems \ref{prime_power} and  \ref{thm3}, we have the following result.
	\begin{cor}\label{cor:2}
		Let $F\in S_{k_1}^\perp(\Gamma_2)$ and $G\in S_{k_2}^\perp(\Gamma_2)$ 
		be Siegel eigenforms with Fourier coefficients 
		$A_{F}(T)$ and $A_{G}(T)$, respectively. 
		Suppose there exists a $T_0$ { such that 
		$-\det (2T_0)=-D_0$ is fundamental, 
		$\mathbb{Q}(\sqrt{-D_0})$ has class number $1$ and} $A_{F}(T_0) = A_{G}(T_0) =1$. Then we have the following.
		\begin{enumerate}
			\item 
			If $F$ and $G$  are not { scalar multiples} of each other, then there exists a set of primes of density $1$ such that for each prime $p$ in this set, the sequence $\{A_F(p^r T_0)A_G(p^rT_0)\}_{r\ge 0}$ has infinitely many sign changes.
			\item
			If
			$$
			\liminf_{x\rightarrow \infty }\frac{\#\{n\le x: {\rm sign}\left( A_F(n T_0)\right)\neq {\rm sign}\left( A_G(n T_0)\right) \}}{x} =0,
			$$
			then $k_1=k_2$ and $F$ is a scalar multiple of $G$.
		\end{enumerate}
	\end{cor}
	
	%\subsection*
	\noindent {\bf{Structure of the paper.}}
	In \S \ref{smf}, we recall Siegel modular forms %, Satake parameters 
	and some properties of the 	spinor $L$-functions.  
	We also review some standard results about the Galois representations attached to Siegel eigenforms. These representations play a vital role in proving the distinctness of Satake parameters of two Siegel eigenforms in \S \ref{intermediate}. Also in the same section, 
	we show that the degree of a certain polynomial, appearing in the Euler product of a certain Dirichlet series, is at most $14$.    
	Finally, by using the properties of $\mathfrak{B}$-free numbers and the results obtained in \S \ref{intermediate}, we prove Theorems \ref{prime_power} and \ref{thm3} in \S \ref{p_prime_power} and \S \ref{p_thm3}, respectively.

	\section{Prerequisite{s}}\label{smf}
	In this section, we recall some basic properties of Siegel modular forms, associated spinor $L$-functions and Galois representations  which are based primarily on \cite{and} and \cite{pit}.
	\subsection{Siegel modular forms}
		The real symplectic unimodular group of degree $2$ is defined by 
	$$
	{\rm Sp}_4(\R)=\{M\in {\rm GL}_4(\mathbb{R}):MJM^t=J \},
	$$
	where $J=\begin{pmatrix} {0_2} & {I_2}\\{-I_2} & {0_2}\end{pmatrix}$,
	$M^t$ denotes the transpose matrix of the matrix $M$, $0_2$ is the $2\times 2$ zero matrix and
	$I_2$ is the $2\times 2$ identity matrix. 
	Let $\Gamma_2:={\rm Sp}_4(\Z)$ be the subgroup of ${\rm Sp}_4(\R)$ { consisting of matrices} 
	with integer entries.
	%For simplicity, we denote the group ${\rm Sp}_4(\Z)$ by $\Gamma_2$.
	%and the full modular group $SL_2(\Z)$ by $\Gamma_1$. 
	For a positive integer $k$, we denote the space of 
	Siegel modular forms (resp. cusp forms) of weight $k$ on the group $\Gamma_2$ by $M_k(\Gamma_2)$ 
	(resp. $S_k(\Gamma_2)$). 
	There is an algebra of Hecke operators acting on the space $M_k(\Gamma_2)$ 
	which preserves $S_k(\Gamma_2)$. A Siegel modular form in $S_k(\Gamma_2)$ is called a 
	{\it Siegel eigenform} if it is a common eigenvector of all the Hecke operators.
	
	If $k$ is an even integer, the {\it Saito-Kurokawa} conjecture asserts the existence of a lifting of any modular form of weight $2k-2$, level $1$ to a Siegel modular form of degree $2$ of weight $k$ on $\Gamma_2$. 
	This conjecture has been proved now due to the work of many mathematicians.
	The image of this lifting is a special subspace of $M_k(\Gamma_2)$ {named Maass space}. Moreover, this Saito-Kurokawa lifting is an isomorphism { from the modular forms space} onto the Maass space mapping Eisenstein series to Eisenstein series, cusp forms to cusp forms, and eigenforms to  eigenforms. 
	
	Note that the space $S_k(\Gamma_2)$ is a Hilbert space {under the Petersson inner product given by 
	\cite[Eq. (1.1.16)]{and}}. We denote by {$S_k^\perp(\Gamma_2)$ the} subspace  which is the orthogonal complement to the Maass space in $S_k(\Gamma_2)$.
	
	\subsection{Satake $p$-parameters and Spinor $L$-functions}
	Fix a Siegel eigenform $F \in S_k(\Gamma_2)$  with $n$-th Hecke eigenvalue $\mu_F(n)$ and $n$-th normalized Hecke eigenvalue $\lambda_F(n):=\frac{\mu_F(n)}{n^{k-3/2}}$.
	For any prime $p$, let $\alpha_{0,p}, \alpha_{1,p}, \alpha_{2,p}$ be the Satake $p$-parameters of $F$. Then we know that $\alpha_{0,p}^2 \alpha_{1,p} \alpha_{2,p}=p^{2k-3}$. We now define the complex numbers
	\begin{equation}\label{actual_satake}
		\beta_{1,p}:=\frac{\alpha_{0,p}}{p^{k-3/2}},\quad \beta_{2,p}:=\frac{\alpha_{0,p}\alpha_{1,p}}{p^{k-3/2}}, \quad \beta_{3, p}:=\frac{\alpha_{0,p}\alpha_{2,p}}{p^{k-3/2}}, \quad \beta_{4, p}:=\frac{\alpha_{0,p}\alpha_{1,p}\alpha_{2,p}}{p^{k-3/2}}
	\end{equation}
	and by abuse of notation, we say that  $\beta_{i,p}$'s are the (normalized) Satake $p$-parameters of $F$. It is clear that
	$\beta_{1,p}=\beta_{4,p}^{-1}$ and $\beta_{2,p}=\beta_{3,p}^{-1}$.
	
	The  spinor $L$-function attached to $F$ is defined by
	\begin{equation*}\label{euler}
		L(s, F, spin) = \prod_{p\ {\rm prime}} L_p(s,F, spin),
	\end{equation*}
	where
	\begin{equation}\label{beta}
		L_p(s,F, spin)=\prod_{1\le i \le 4}(1-\beta_{i,p}p^{-s})^{-1}.
	\end{equation} 
	Indeed, one can obtain that (see, for example \cite[Theorem 3.11]{pit}) 
	\begin{align*}
		L_p(s,F, spin)^{-1}	 =p^{-4s} - \lambda_F(p)p^{-3s} + \left(\lambda_F(p)^2-\lambda_F(p^2)-p^{-1}\right)  p^{-2s} - \lambda_F(p) p^{-s} 	+1 .
	\end{align*}
	
	\begin{rmk}\label{satake_roots}
		{\rm 	The above equation shows that  $p^{k-3/2}\beta_{i,p}$'s are the roots of the   $p$-th Hecke polynomial (cf. \eqref{hecke_polynomial}). In other words, 
			$$
		\prod_{1\le i \le 4}(X-p^{k-3/2}\beta_{i,p})=	X^4 - \mu_F(p) X^3+ \left(\mu_F(p)^2-\mu_F(p^2)-p^{2k-4}\right) X^2 - \mu_F(p)p^{2k-3} X 	+ p^{4k-6}.
			$$
			This fact will play a crucial role in the proof of \thmref{satake}.}
	\end{rmk}
	
	%By abuse of notation, we say that the complex numbers $\beta_{1,p},\beta_{2,p},\beta_{3,p},\beta_{4,p}$ are the Satake $p$-parameters of $F$ which are the roots of the Hecke polynomials at $p$ after suitable normalization.
	
	\subsection{Bounds of eigenvalues}
	We now assume that the eigenform $F\in S_k^\perp(\Gamma_2)$. Then the generalized Ramanujan conjecture proved by Weissauer \cite{wei09}  asserts that for any prime $p$, we have
	\begin{equation}\label{rama}
		|\beta_{i,p}|=1 \quad {\rm for~all ~} 1\le i\le 4.
	\end{equation}
	However,  the above assertion is not true if $F$ is a Saito-Kurokawa lift. Writing 
	$\displaystyle{L(s, F, spin) = \sum_{n=1}^\infty\frac{a_F(n)}{n^s}}$ as a Dirichlet series, one deduces from \eqref{rama} that 
	\begin{equation}\label{bound1}
		|a_F(n)|\le d_4(n) \quad {\rm for~all}\quad n\ge 1,
	\end{equation}
	where $d_4(n)$  is the number of ways of writing $n$ as a product of $4$ positive integers.
	Furthermore,  it easily follows from \cite[Theorem 1.3.2]{and} that
	\begin{equation}\label{mobius1}
		\sum_{n=1}^\infty\frac{\lambda_F(n)}{n^{s}}=\zeta(2s+1)^{-1}L(s, F, spin) = \zeta(2s+1)^{-1}\sum_{n=1}^\infty\frac{a_F(n)}{n^s}
	\end{equation}
	from which we conclude that for any integer $n\ge 1$, we have
	\begin{equation*}\label{mobius2}
		{\lambda_F(n)}=\sum_{d^2|n}\frac{\mu(d)}{d}a_F\left({n}/{d^2}\right),
	\end{equation*}
	where $\mu$ is the M\"{o}bius function.
	Using the bound \eqref{bound1}, we see that $|\lambda_F(n)|\le \sum_{d^2|n} \frac{d_4(n/d^2)}{d}$ and therefore for any $\epsilon>0$,  we have
	\begin{equation}\label{raman}
		|\lambda_F(n)| \ll_\epsilon n^{\epsilon}.
	\end{equation}

	\subsection{Galois representations attached to Siegel modular forms}
	Let $F\in S_{k}(\Gamma_2)$ be a Siegel eigenform with $n$-th
	eigenvalue $\mu_F(n)$. Suppose $E$ denotes 
	the number field generated by $\mu_F(n)$, $ n\ge 1$. 
	It follows from the works of Taylor, 
	% \cite{tay}, 
	Laumon 
	%\cite{lau} 
	 and Weissauer 
	 %\cite{wei05} 
	 that one can attach a family of  Galois representations  to
	$F$. More precisely, if $\lambda$ is a prime in $E$ above a rational prime $\ell$ and  $E_\lambda$ denotes the completion of $E$ at $\lambda$, then there
	exists a continuous semisimple Galois representation  
	$$\rho_{F,\lambda}: {\rm Gal}(\overline{\Q}/\Q) \rightarrow {\rm GSp}_4(\overline{E}_\lambda)$$
	such
	that $\rho_{F,\lambda}$ is unramified outside $\ell$.  Moreover, the characteristic polynomial of 
	the Frobenius ${\rm Frob}_p$ at $p \neq \ell$ is  
	\begin{equation}\label{hecke_polynomial}
		X^4 - \mu_F(p) X^3 + \left(\mu_F(p)^2- \mu_F(p^2) - p^{2k-4} \right) X^2 - p^{2k-3}\mu_F(p) X
		+p^{4k-6}.
	\end{equation}
Furthermore, for all but finitely many primes $\ell$, the representation $\rho_{F,\lambda}$ is valued in  ${\rm GSp}_4(E_\lambda)$ (see, for example \cite[Corollary 2.2]{kkw-lt}) and hence we can view $\rho_{F,\lambda}$ (after conjugating this representation) as a representation valued in ${\rm GSp}_4(\mathcal{O}_{E_\lambda})$, 	where $\mathcal{O}_{E_\lambda}$ is the ring of integers of $E_\lambda$.
%	If $\rho_{F,\lambda}$ is irreducible, then it is defined over $E_\lambda$, the completion of $E$ with respect to $\lambda$.
	%	Weissauer \cite[Theorem II]{wei05} has proved that the Galois representation attached to a Siegel eigenform which is a Saito-Kurokawa lift is a direct sum of two 	Galois representations each of dimension two, and hence it is always reducible. 
	%However, 
	%In fact,	Ramakrishnan \cite[Theorem B]{ram} proved that if  $F\in S_k^\perp(\Gamma_2)$ then the	representation $\rho_{F,\lambda}$   is	irreducible for all $\ell > 4k-5$.
	%	Under the hypothesis of irreducibility, 
Therefore, the maximal
	possible image for $\rho_{F,\lambda}$ is the group
	$$
	A_\lambda= \{\gamma \in  {\rm GSp}_4(\mathcal{O}_{E_\lambda}) : {\rm sim}(\gamma) \in {(\Z_\ell^\times)}^{2k-3}\},
	$$
	where ${\rm sim}(\gamma)$ denotes the similitude of $\gamma$.
	%If the image of $\rho_{F,\lambda}$ is $A_\lambda$, then we say that the image is {\em large} or {\em big}.
%	From the work of Dieulefait \cite[Theorem 4.2]{die02} (see, also \cite[Theorem 1.4 ($ii$)]{kkw-lt}) we know that if  $F$ is not a Saito-Kurokawa lift and $\ell$ is a sufficiently large prime, 
	%which splits completely in $E$,
%	 then the image of  $\rho_{F,\lambda}$ is $A_\lambda$. 
Let $F$ be a non-Saito-Kurokawa lift.
	 Dieulefait \cite[Theorem 4.2]{die02} has proved that if $\ell$ is large enough and splits completely in $E$, then the image of  $\rho_{F,\lambda}$ is $A_\lambda$. In a recent paper \cite{kkw-lt}, the precise image of the Galois representations attached to  Siegel eigenforms of higher level has been studied. Since $F$ is of level $1$, it cannot have any (non-trivial) inner twists. { This can be easily deduced from the fact that if $(\sigma,\chi)$ is an inner twist of a Siegel eigenform of level $N$ and nebentype $\varepsilon$ then we have $ \chi^2=\sigma(\varepsilon)\cdot \varepsilon^{-1}$ and hence the
	 	only prime divisors of the conductor of $\chi$ are the primes dividing $N$ (see \cite[Section 2.1]{kkw-lt})}. 
		So, it follows from \cite[Theorem 1.4 ($ii$)]{kkw-lt}) that for any sufficiently large prime $\ell$, 
the image of  $\rho_{F,\lambda}$ is $A_\lambda$.

	For our purpose, we need the  information about the images of the product Galois representations which we recall from \cite{kkw}. Let $F$ and $G$ be  Siegel eigenforms  and 
	%with eigenvalues $\mu_F(n)$ and $\mu_G(n)$ respectively 
%	which are not a constant multiple of each other and neither of them is a Saito-Kurokawa lift. 
	let $E$ be the number field generated by  all the Hecke eigenvalues of $F$ and $G$. Let $\lambda$ be a prime in $E$ above a rational prime $\ell$.  Let $\rho_{F, \lambda}$ and $\rho_{G, \lambda}$ be the $\lambda$-adic Galois representations attached to $F$ and $G$. From the above discussion, we may assume that, for a large prime $\ell$, both of these  representations are valued in ${\rm GSp}_4(\mathcal{O}_{E_\lambda})$. 	 Consider the product  Galois representation 
	$$ \rho_{F, \lambda}\times \rho_{G,\lambda}: {\rm Gal}(\overline{\Q}/\Q) \longrightarrow {\rm GSp}_4(\mathcal{O}_{E_\lambda}) \times {\rm GSp}_4(\mathcal{O}_{E_\lambda}) \quad {\rm defined ~by}	\quad  
	\sigma\mapsto (\rho_{F, \lambda}(\sigma), \rho_{G, \lambda}(\sigma)).
	$$
   Then we have the following result.
	\begin{thm}\cite{kkw}\label{product_image}
	Let $F \in S_{k_1}^\perp(\Gamma_2)$ and $G \in S_{k_2}^\perp(\Gamma_2)$ be Siegel eigenforms such that $F$ is not a constant multiple of $G$. Then for all but finitely many primes  $\ell$, the image of 	$ \rho_{F, \lambda}\times \rho_{G,\lambda}$ is 
		$$
		\{(\gamma_1,\gamma_2) \in  {\rm GSp}_4(\mathcal{O}_{E_\lambda}) \times {\rm GSp}_4(\mathcal{O}_{E_\lambda}) : {\exists ~v \in \Z_\ell^\times}, {\rm sim}(\gamma_i)=v^{2k_i-3}, 1\le i\le 2\}.
		$$
	\end{thm}
	
	\subsection{Algebraic Chebotarev density theorem}
	We now recall the following result, a special case of a result of Rajan \cite{raj}, giving an algebraic formulation of the Chebotarev density	theorem  for the density of primes satisfying an algebraic conjugacy condition. This plays a crucial role in proving \thmref{satake}.
	\begin{thm}\cite[Theorem 3]{raj}\label{rajan_result}
		Let $K$ be a finite extension of $\Q_\ell$ and let $\mathcal G$ be
		an algebraic group over $K$. Let $\mathcal X$ be a subscheme of $\mathcal G$ defined over $K$ that is stable
		under the adjoint action of $\mathcal G$. Suppose that 
		$$ R : {\rm Gal}(\overline{\Q}/\Q) \rightarrow \mathcal G(K)$$
		is a Galois representation which is unramified outside a finite set of primes.
		% and let 		$ \mathcal C = \mathcal X(K) \cap \rho({\rm Gal}(\overline{\Q}/\Q))$.
		Let $\mathcal H$ denote the Zariski closure of  ${\rm Im}(R)$ $($the image of $R)$ in $\mathcal G(K)$, with identity connected
		component $\mathcal H^\circ$ and component group $\Phi = \mathcal H/\mathcal H^\circ$. For each $\phi \in \Phi$, assume that $\mathcal H^\phi$ denotes its
		corresponding connected component. Let
		$$
		\Psi=\{\phi \in \Phi:\mathcal H^\phi \subset \mathcal X\}.
		$$
		Then the set of primes $p$  with $R({\rm Frob}_p) \in \mathcal X(K) \cap {\rm Im}(R)$ has density $\frac{|\Psi|}{|\Phi|}$.
	\end{thm}
	
	\subsection{$\mathcal{B}$-free numbers}		
	The notion of $\mathcal{B}$-free numbers was first introduced by Erd\"{o}s in \cite{erdos}. These numbers are a certain generalization of squarefree numbers.	
	\begin{defn} 
		Let $\mathcal{B}=\{b_i: 1<b_1<b_2<\cdots\}$ be a sequence of positive integers such that 
		$$
		\sum_{i=1}^{\infty}\frac{1}{b_i}<\infty~~~{and}~~~{\rm gcd}(b_i, b_j)=1~~~{for}~~~i\neq j.
		$$
		A positive integer $n$ is said to be $\mathcal{B}$-free if it is not divisible by $b_i$ for any $i\ge 1$.
	\end{defn}	
	
	Let $\mathcal{A}$ be the set of all $\mathcal{B}$-free numbers, then a result of Erd\"{o}s \cite[Theorem 3]{erdos} states that
	\begin{equation}\label{bfree}
		\#\{n\le x : n\in \mathcal{A}\}\sim \delta x~~~{\rm as}~~~x\rightarrow \infty,
	\end{equation}
	where $\delta=\displaystyle\prod_{b_i\in \mathcal{B}}\left(1-\frac{1}{b_i}\right)$.	
	Note that if we take $\mathcal{B}$ to be the sequence of squares of all primes, then the set of $\mathcal{B}$-free numbers is nothing but the set of all squarefree numbers.

	\section{Technical Results}\label{intermediate}
	%We first state the following two strong multiplicity one results proved in \cite{kms}.	
	%We also recall the following result of Kowalski and Saha on the estimate on the number of vanishing Hecke eigenvalues at prime indices.
	%	\begin{thm}\label{main3}
	%	Let $F\in S_k(\Gamma_2)$be an eigenform with $n$-th eigenvalues $\mu_F(n)$. Then there exists $\delta>0$ such that
	%	$$
	%	\#\{p\le x: \mu_F(p)=0\}\ll \frac{x}{(\log{x})^{1+\delta}} \cdot
	%	$$
	%	\end{thm}

	% so that  the  spinor $L$-functions associated to $F$ and $G$ be respectively given by
%	$$
%	L(s, F, spin)=\prod_p\prod_{1\le i \le 4}(1-\beta_{i,p}p^{-s})^{-1} \quad \mbox{and}\quad	L(s, G, spin)=\prod_p\prod_{1\le i \le 4}(1-\delta_{i,p}p^{-s})^{-1}.
%	$$
%	In this section, we prove  two results about these Satake parameters and $L$-functions.  
	\subsection{Distinctness of Satake $p$-parameters}
	We start by recalling a result of Weiss   about the distinctness of Satake $p$-parameters of a non-Saito-Kurokawa Siegel eigenform. For $F\in S_{k}^\perp(\Gamma_2)$ 
	with  Satake $p$-parameters $\{\beta_{i,p}\}_{ 1\le i \le 4}$, defined by \eqref{actual_satake}, he  proved that
	there exists a set of primes of density 1 such that for any prime $p$ in that set, all the four elements $\beta_{i,p}, {1\le i\le 4},$ are distinct \cite[Corollary 5.11]{wei}. Using a similar argument, we prove the following result about the distinctness of Satake $p$-parameters for two Siegel eigenforms.
	% to prove  that for a density 1 set primes $p$, these two sets also do not intersect. 
	% This result plays an important role to establish the 	simultaneous sign changes result of the eigenvalues in the next section.
	
	% attached to Siegel eigenforms recalled in the previous section. 
	\begin{thm}\label{satake}
		Let $F\in S_{k_1}^\perp(\Gamma_2)$ and $G\in S_{k_2}^\perp(\Gamma_2)$ be Siegel eigenforms which are not a constant multiple of each other. For a prime $p$, let $\{\beta_{i,p}\}_{ 1\le i \le 4}$ and $\{\delta_{i,p}\}_{ 1\le i \le 4}$ be the Satake $p$-parameters of $F$ and $G$, respectively. Then there exists a set of primes of density $1$ such that for any prime $p$ in that set, all the eight elements 
		$$ 
		\{\beta_{i,p}, \delta_{j,p} : 1\le i,j \le 4\}
		$$
		are distinct.
	\end{thm}
	
	\begin{proof}
		Let $E$ be the number field generated by the eigenvalues of $F$ and $G$. Let  $\lambda$ be a prime in $E$ above a large rational prime $\ell$ such that the assertion of \thmref{product_image} is true. In particular, both the representations $\rho_{F, \lambda}$ and $\rho_{G, \lambda}$ are valued in  ${\rm GSp}_4(\mathcal O_{E_\lambda})$. We now consider 
		{the} $\lambda$-adic product  Galois representation 
		$$
		R_\lambda:= \rho_{F, \lambda} \times  \rho_{G,\lambda}: {\rm Gal}(\overline{\Q}/\Q) \longrightarrow \mathcal G (\mathcal O_{E_\lambda}), 
		$$
	%	defined by 
	%	$$
	%	R_\lambda(\sigma)=\left(\rho_{F, \lambda}(\sigma)\chi_\ell^{{3/2-k_1}}(\sigma), \rho_{G, \lambda}(\sigma)\chi_\ell^{{3/2-k_2}}(\sigma)\right),
	%	$$
	%	where $\chi_\ell$ is the $\ell$-adic cyclotomic character and the algebraic group  
	where $\mathcal G :={\rm GSp}_4\times {\rm GSp}_4$ { is the algebraic group considered over $\overline \Q_\ell$}.
				%This representation is well-defined because $\det(\rho_{F, \lambda})=\chi_\ell^{{4k_1-6}}$ and $\det(\rho_{G, \lambda})=\chi_\ell^{{4k_2-6}}$.
		For any prime $p\neq\ell$, if
		$$
		R_\lambda({\rm Frob}_p)=(\gamma_1, \gamma_2),
		$$
		then by \rmkref{satake_roots}, it is easy to observe that the eigenvalues of ${\rm sim}(\gamma_2)\gamma_1^2$
	and	${\rm sim}(\gamma_1)\gamma_2^2$  are 
		$$
		p^{2(k_1+k_2-3)}\beta_{1,p}^2, \quad 	p^{2(k_1+k_2-3)}\beta_{2,p}^2, \quad 	p^{2(k_1+k_2-3)}\beta_{3,p}^2, \quad 	p^{2(k_1+k_2-3)}\beta_{4,p}^2;
		$$ 
		and
			$$
		p^{2(k_1+k_2-3)}\delta_{1,p}^2, \quad 	p^{2(k_1+k_2-3)}\delta_{2,p}^2, \quad 	p^{2(k_1+k_2-3)}\delta_{3,p}^2, \quad 	p^{2(k_1+k_2-3)}\delta_{4,p}^2,
		$$
		respectively.   Hence in order to complete the proof,  it is sufficient to show that for a set of primes $p$ of density 1, any two elements in the collection of all the eight eigenvalues of ${\rm sim}(\gamma_2)\gamma_1^2$
		and	${\rm sim}(\gamma_1)\gamma_2^2$  are distinct, where 
		$R_\lambda({\rm Frob}_p)=(\gamma_1, \gamma_2).
		$
		
	{ Let $\mathcal X$ be the set of elements $(\gamma_1, \gamma_2)$ of $\mathcal G$  
			such that} at least two eigenvalues  are the same among  all the eight  eigenvalues of ${\rm sim}(\gamma_2)\gamma_1^2$
			and	${\rm sim}(\gamma_1)\gamma_2^2$.
			%	Let $\mathcal X$ be the set of elements $(\gamma_1, \gamma_2)$ of $\mathcal G~ (=\mathcal G(\overline \Q_\ell))$
		%:= {\rm Sp}_4(\overline \Q_\ell) \times {\rm Sp}_4(\overline \Q_\ell)$ 
	%	such that at least two eigenvalues  are the same among  all the eight  eigenvalues of ${\rm sim}(\gamma_2)\gamma_1^2$ and	${\rm sim}(\gamma_1)\gamma_2^2$.
		Then we need to show that the set of primes $p$ such that $R_\lambda(\mathrm{Frob}_p)\in  { \mathcal X({E_\lambda})}\cap {\rm Im}(R_\lambda)$ is of density $0$, { where $\mathcal X({E_\lambda})$ is the set of $E_\lambda$-points of $\mathcal X$}.
		We will prove it by applying \thmref{rajan_result} in this setting and for that, we first need to check certain required properties of $\mathcal X$ and  understand the image of $R_\lambda$.
		
		Clearly, $\mathcal X$ is stable
		under the conjugate action of $\mathcal G$ because the eigenvalues do not depend on conjugacy classes. {Also, $\mathcal{X}$} is a closed subscheme of $\mathcal G$. To see this, consider the polynomial 
		$$f(X) = \mathcal P\left({\rm sim}(\gamma_2)\gamma_1^2\right)(X) \cdot \mathcal P\left({\rm sim}(\gamma_1)\gamma_2^2\right) (X),$$
		 where $ \mathcal P(\gamma)(X)$ denotes the characteristic polynomial of $\gamma \in {\rm GSp}_4$. Then $\mathcal X$ is the vanishing set of the discriminant  of the polynomial $f$. 
		
			Let $\mathcal H:=\overline{{\rm Im}(R_\lambda)}$ {denote} the Zariski closure of  ${\rm Im}(R_\lambda)$. 
		From \thmref{product_image}, it follows that the algebraic group $\mathcal H$ is connected which might be well known to the experts but for the sake of completeness, we have given an idea of its proof in \lemref{dense}.

		Summarizing the above discussions, we have
		\begin{itemize}
			\item
			a representation $R_\lambda: {\rm Gal}(\overline{\Q}/\Q) \rightarrow {\mathcal G({E_\lambda})}$ such that  the Zariski closure of its image (denoted by $\mathcal H$) is a connected {algebraic} group; and
			\item
			a closed subscheme $\mathcal X$ of $\mathcal G$ that is stable under the action of $\mathcal G$ by conjugation.
		\end{itemize}
		We are now ready to apply \thmref{rajan_result}.
		Since $\mathcal H$ is connected, the identity connected component $\mathcal H^\circ$ is equal to $\mathcal H$ and hence the component group $\Phi$ is trivial. Also, it is obvious that  $\mathcal H \not\subset \mathcal X$ which gives $|\Psi|=0$. Applying \thmref{rajan_result}, we deduce that the set of primes $p$ such that $R_\lambda(\mathrm{Frob}_p)\in {\mathcal X({E_\lambda})}\cap {\rm Im}(R_\lambda)$ has density $0$. This completes the proof.
	\end{proof}
	
	Finally, we give a proof of the fact that the algebraic group $\mathcal H$, defined in the proof of the above theorem, is connected. 
	%Let $\mathcal H:=\overline{{\rm Im}(R_\lambda)}$ denotes the Zariski closure of  ${\rm Im}(R_\lambda)$. We first prove the following lemma which might be well known to the experts but for the sake of completeness, we give an idea of the proof.
	\begin{lem}\label{dense}
		The algebraic group $\mathcal H$ is {homeomorphic} to ${\rm GSp}_4 \times {\rm Sp}_4$. In particular, $\mathcal H$ is connected.
	\end{lem}
	\begin{proof}
		From \thmref{product_image}, recall that
		$$
	{\rm Im}(R_\lambda)=	\{(\gamma_1,\gamma_2) \in  {\rm GSp}_4(\mathcal{O}_{E_\lambda}) \times {\rm GSp}_4(\mathcal{O}_{E_\lambda}) : { \exists ~v \in \Z_\ell^\times}, {\rm sim}(\gamma_i)=v^{2k_i-3}, 1\le i\le 2\}.
		$$
		Let 
		$
		A_{\lambda,1}= \{\gamma \in  {\rm GSp}_4(\mathcal{O}_{E_\lambda}) : {\rm sim}(\gamma) \in {(\Z_\ell^\times)}^{2k_1-3}\}.
		$  
		We now consider the following split short exact sequence:
		$$
		\begin{tikzcd}
			1\arrow{r} & {	A_{\lambda,1}}\arrow{r}{\xi} & {\rm Im}(R_\lambda) \arrow{r}\arrow[bend left=33]{l}{\nu} & \faktor{{\rm Im}(R_\lambda)}{\xi\left(	A_{\lambda,1}\right) }\arrow{r} & 1.
		\end{tikzcd}
		$$
		In the above sequence, for $\gamma_1\in A_{\lambda,1}$ with ${\rm sim}(\gamma_1)=v^{2k_1-3}$, the injective homomorphism $\xi$ is defined by 
		$$
		\xi(\gamma_1)= (\gamma_1,\gamma_2), {\rm~ where}~ \gamma_2=J(v,k_2):=\begin{pmatrix} v^{2k_2-3}{I_2} & {0_2}\\{0_2} & {I_2}\end{pmatrix}
		$$
		and the section $\nu$ is the natural projection. Therefore, we have $ {\rm Im}(R_\lambda) \simeq	A_{\lambda,1} \times \faktor{{\rm Im}(R_\lambda)}{\xi\left(	A_{\lambda,1}\right) }$. From \cite[Proof of Theorem 1.5]{kms}, we know that  ${A_{\lambda,1}}$ is  Zariski dense in ${\rm GSp}_4$ and so by using the fact that {a} product of closed spaces is closed in the product topology and $\mathcal H$ is the Zariski closure of 	${\rm Im}(R_\lambda)$, we obtain
		$$
		\mathcal H \simeq 	{\rm GSp}_4 \times \overline{\left(\faktor{{\rm Im}(R_\lambda)}{\xi\left(	A_{\lambda,1}\right) }\right)}.
		$$
		By \cite[Lemma 3]{prasad}, we know that ${\rm Sp}_4(\mathcal{O}_{E_\lambda})$ is Zariski dense in ${\rm Sp}_4$ and hence to complete the proof it is sufficient to show that 
		$$\faktor{{\rm Im}(R_\lambda)}{\xi\left(	A_{\lambda,1}\right) } \simeq {\rm Sp}_4(\mathcal{O}_{E_\lambda}).$$
		Fix a coset $(\gamma_1,\gamma_2)\xi\left(	A_{\lambda,1}\right)$ in  the quotient group  with ${\rm sim}({\gamma_i})=v^{2k_i-3}$ for some $v\in \Z_\ell^\times$. Then we can write
		$$
		(\gamma_1,\gamma_2)\xi\left(	A_{\lambda,1}\right)= (\gamma_1,\gamma_2)(\gamma_1^{-1}, J(v,-k_2) ) \xi\left(	A_{\lambda,1}\right)= (I_4, \gamma_2J(v,-k_2))\xi\left(	A_{\lambda,1}\right).
		$$
		It is now easy to see that the map $(\gamma_1,\gamma_2)\xi\left(	A_{\lambda,1}\right) \mapsto \gamma_2J(v,-k_2)$ is a well-defined injective homomorphism onto ${\rm Sp}_4(\mathcal{O}_{E_\lambda})$, which completes the proof.
	\end{proof}
	
	%The remaining part of this section will be only needed to prove \thmref{prime_power}.
	
	\subsection{Euler product of the spinor $L$-functions}	
	The aim of this section is to show that the degree of the polynomial, appearing in the Euler product of a certain Dirichlet series with coefficients the product of eigenvalues at prime powers for a fixed prime, is at most 14. This result improves on a result of Gun {\em et al.} \cite[Lemma 15]{gkp}. %which will be used to prove \thmref{prime_power}. 
	Though the result of Gun {\em et al.} is sufficient for our purpose, we would like to present the proof of our result since it gives a better result by following a completely different 
	but elementary approach compared to \cite[Lemma 15]{gkp}. We follow the idea presented in the proof of \cite[Lemma 2.7.13]{deitmar}.
	\begin{lem}\label{14}	
		Let $F\in S_{k_1}^\perp(\Gamma_2)$ and $G\in S_{k_2}^\perp(\Gamma_2)$ be Siegel eigenforms with respective $n$-th normalized eigenvalues $\lambda_F(n)$ and $\lambda_G(n)$.  Let $\{\beta_{i,p}\}_{ 1\le i \le 4}$ and $\{\delta_{i,p}\}_{ 1\le i \le 4}$ be the Satake $p$-parameters of $F$ and $G$, respectively. 	Then for any prime $p$ and any $s$ with ${\rm Re}(s)>0$, we have
		$$
		\sum_{r=0}^\infty \frac{\lambda_F(p^r) \lambda_G(p^r)}{p^{rs}}= g_p(p^{-s}) \prod_{1\le i, j\le 4}(1-\beta_{i,p}\delta_{j,p}p^{-s})^{-1},
		$$
		where $g_p(p^{-s})$ is a polynomial in $p^{-s}$ of degree at most $14$.
	\end{lem}
	\begin{proof}
		By using \eqref{mobius1} and \eqref{raman}, we deduce that for each prime $p$ and ${\rm Re}(s)>0$, we have
		\begin{equation*}\label{taylor1}
			\Phi_F(p^{-s}):= \sum_{r=0}^\infty\frac{\lambda_F(p^r)}{p^{rs}}=\left(1-\frac{1}{p^{2s+1}}\right)\prod_{1\le i \le 4}(1-\beta_{i,p}p^{-s})^{-1},
		\end{equation*}
		and
		\begin{equation*}\label{taylor2}
			\Phi_G(p^{-s}):= \sum_{r=0}^\infty\frac{\lambda_G(p^r)}{p^{rs}}=\left(1-\frac{1}{p^{2s+1}}\right)\prod_{1\le i \le 4}(1-\delta_{i,p}p^{-s})^{-1}.
		\end{equation*}
		Let $s_0=\sigma_0+it_0$ be any complex number such that $\sigma_0>0$.
		Let $C$ be the circle centred at $0$ of radius bigger than $1$ but strictly less than $2^{\sigma_0}$. Now consider the integral
		\begin{equation}\label{cauchy}
			\frac{1}{2\pi i}\int_C \Phi_F(p^{-s_0}z) \Phi_G(1/z)\frac{dz}{z}.
		\end{equation}
		Using the above series form of $\Phi_F$ and $\Phi_G$, 
		we deduce that the above integral \eqref{cauchy} is equal to
		\begin{equation}\label{one}
			\sum_{r, r'=0}^\infty \lambda_F(p^r) \lambda_G(p^{r'})p^{-s_0r}\frac{1}{2\pi i} \int_C z^{r-r'-1}dz
			=\sum_{r=0}^\infty \lambda_F(p^r) \lambda_G(p^r)p^{-s_0r}.
		\end{equation}
		By using the Euler product form of $\Phi_F$ and $\Phi_G$,
		the integral \eqref{cauchy} is equal to 
		\begin{equation*}\label{another}
			\frac{1}{2\pi i}\int_C \frac{\left(1-\frac{1}{p}p^{-2s_0}z^2\right)\left(1-\frac{1}{pz^2}\right)}
			{\displaystyle\prod_{1\le i\le 4}(1-\beta_{i,p}p^{-s_0}z)\prod_{1\le i \le 4}\left(1-\frac{\delta_{i,p}}{z}\right)}\frac{dz}{z}.
		\end{equation*}
		Since $|\beta_{i, p}|=|\delta_{i, p}|=1$ for $1\le i\le 4$, the singularities of the integrand of the above integral which are inside $C$ are the poles at $z=\delta_{i, p}$ for $1\le i\le 4$. Applying Cauchy's residue theorem, we deduce that
		\begin{equation}\label{final}
			\frac{1}{2\pi i}\int_C \Phi_F(p^{-s_0}z) \Phi_G(1/z)\frac{dz}{z}
			=g_p(p^{-s_0}) \prod_{1\le i, j\le 4}(1-\beta_{i,p}\delta_{j,p}p^{-s_0})^{-1},
		\end{equation}
		where $g_p(p^{-s_0})$ is a polynomial in $p^{-s_0}$ of degree at most $14$.
		Therefore from \eqref{one} and \eqref{final}, we conclude the result.
	\end{proof}

	\section{Proof of \thmref{prime_power}}\label{p_prime_power}
	
	For a prime $p$, let $\{\beta_{i,p}\}_{ 1\le i \le 4}$ and $\{\delta_{i,p}\}_{ 1\le i \le 4}$ be the respective Satake $p$-parameters of $F$ and $G$, defined by \eqref{actual_satake}. Let $\mathcal P$ be the set of primes $p$ such that any two elements in the set 
	$$ 
	\{\beta_{i,p}, \delta_{j,p} : 1\le i,j \le 4\}
	$$
	are distinct. From \thmref{satake}, we know that the set of primes $\mathcal P$  is of density 1. To prove \thmref{prime_power}, we show that for any $p\in \mathcal P$,  the sequence 
	$\{\lambda_F(p^r)\lambda_G(p^r)\}_{r\ge 0}$ changes sign infinitely often.
	
	We now fix a prime $p\in \mathcal P$ and on the contrary, we assume that the sequence 
	$\{\lambda_F(p^r)\lambda_G(p^r)\}_{r\ge 0}$ has all but finitely many  terms non-negative. 	Since $F$ and $G$ are non-Saito-Kurokawa lifts, by using \eqref{raman} we see that the Dirichlet series
	\begin{equation}\label{landau}
		\sum_{r=0}^\infty \frac{\lambda_F(p^r) \lambda_G(p^r)}{p^{rs}}
	\end{equation}
	converges absolutely for ${\rm Re}(s)>0$.
	Also, from \lemref{14}, we know that
	\begin{equation}\label{quotient}
		\sum_{r=0}^\infty \frac{\lambda_F(p^r) \lambda_G(p^r)}{p^{rs}}= g_p(p^{-s}) \prod_{1\le i, j\le 4}(1-\beta_{i,p}\delta_{j,p}p^{-s})^{-1},
	\end{equation}
	where $g_p(p^{-s})$ is a polynomial in $p^{-s}$ of degree at most $14$. 
	Since $|\beta_{i,p}|=1=|\delta_{j,p}|$ for all $1\le i, j\le 4$, the function 
	$$
	\prod_{1\le i, j\le 4}(1-\beta_{i,p}\delta_{j,p}p^{-s})^{-1}
	$$
	has 16 singularities on the line ${\rm Re}(s)=0$. On the other hand, as a polynomial in $p^{-s}$, $g_p(p^{-s})$ is of degree at most $14$ showing that the function on the right of \eqref{quotient} and hence the Dirichlet series \eqref{landau} has at least two singularities on the line ${\rm Re}(s)=0$. Therefore the abscissa of absolute convergence of the Dirichlet series \eqref{landau} is $0$. Thus by Landau's theorem {\cite[Theorem 11.13]{Apostol}} on Dirichlet series with non-negative coefficients, we deduce that the Dirichlet series \eqref{landau} has a singularity at $s=0$ and thus the series on the right of \eqref{quotient} has a singularity at $s=0$. It follows that $
	\prod_{1\le i, j\le 4}(1-\beta_{i,p}\delta_{j,p})=0$ and hence
	$$
	\beta_{i,p}=\delta_{j,p}^{-1} \quad \text{for ~some}\quad i,j.
	$$
	But we know that  $\delta_{1, p}^{-1}=\delta_{4, p}$,
	$\delta_{2, p}^{-1}=\delta_{3, p}$ and hence we arrive at a contradiction because of the  choice of our prime  $p$.
	
	{\section{Proof of \thmref{thm3}}\label{p_thm3}
		Consider the set
		$$\mathcal{F}:=\{n\in \mathbb{N} : \lambda_F(n)\lambda_G(n)<0\}.$$ 
		Suppose $F$ is not a scalar multiple of $G$.	To prove the theorem, we show that
		$$
		\#\{n\le x: n\in \mathcal{F}\}\gg x,	
		$$
{ that is, there exists an absolute constant $c > 0$ such that $\#\{n\le x: n\in \mathcal{F}\}\ge c ~x$ for all sufficiently large $x$}. 
By \thmref{prime_power}, there exists a prime $p_0$ and an integer $t\ge 1$ such that $\lambda_F(p_0^t)\lambda_G(p_0^t)<0$. Define
		$$
		\mathcal{B}:=\{p_0\}\cup \{p:  p\neq p_0 {\rm ~and~} \lambda_F(p) \lambda_G(p)=0  \}
		\cup \{p^2:  p\neq p_0 {\rm ~and~} \lambda_F(p) \lambda_G(p)\neq 0 \},
		$$
		and let $\mathcal{A}$ be the set of all $\mathcal{B}$-free numbers. Then any element $n$ of $\mathcal{A}$ is a squarefree positive integer satisfying $\lambda_F(n)\lambda_G(n)\neq 0$.
		We now claim that
		\begin{equation*}\label{finite}
			\sum_{b\in \mathcal{B}}\frac{1}{b}<\infty.
		\end{equation*}
		Since $\displaystyle{\sum_{p}\frac{1}{p^2}<\infty}$, to prove the above claim it is sufficient to show that
		$
		\displaystyle{\sum_{ \lambda_F(p)\lambda_G(p)=0}\frac{1}{p}<\infty}.
		$
		But
		$$
		\sum_{ \lambda_F(p)\lambda_G(p)=0}\frac{1}{p}\le
		\sum_{\lambda_F(p)=0}\frac{1}{p}+ \sum_{\lambda_G(p)=0}\frac{1}{p}.
		$$ 
		We know from \cite[Theorem 4]{rsw} that  there exists a $\delta>0$ such that
		$$
		\#\{p\le x: \lambda_F(p)=0\}\ll \frac{x}{(\log{x})^{1+\delta}}.
		$$
		Using the above result, integration by parts yields 
		$$
		\sum_{p\le x \atop \lambda_F(p)=0}\frac{1}{p}=
		\int_{2^{-}}^x\frac{1}{u}d\left(\sum_{p\le u \atop \lambda_F(p)=0} 1\right)\ll 1+\int_{2}^x \frac{du}{u(\log{u})^{1+\delta}}\ll 1.$$
		Similarly, we have
		$$
		\sum_{p\le x \atop \lambda_G(p)=0}\frac{1}{p}\ll 1.
		$$
		This proves our claim.

		Note {that the} set $\mathcal{A}$ can be written as a disjoint union of the sets $\mathcal{A}_1$ and $\mathcal{A}_2$ 
		defined by
		$$
		\mathcal{A}_1=\{n\in \mathcal{A}: \lambda_F(n)\lambda_G(n)<0\}, \quad 
		\mathcal{A}_2=\{n\in \mathcal{A}: \lambda_F(n)\lambda_G(n)>0\}.
		$$	
		If $n\in \mathcal{A}_2$, then $p_0^tn\in \mathcal{F}$. Therefore 
		$\mathcal{F}\supseteq\mathcal{A}_1\cup p_0^t \mathcal{A}_2$ and  hence we have
		$$
		\#\{n\le x: n\in \mathcal{F}\}\ge \#\{n\le x: n\in \mathcal{A}_1 \cup p_0^t\mathcal{A}_2\}
		\ge  \#\{n\le x/p_0^t: n\in \mathcal{A}\}.
		$$
		Next by using \eqref{bfree} { 
		we deduce that $\#\{n\le x/p_0^t: n\in \mathcal{A}\} \sim \delta ~x/p_0^t$, for a constant 
		$\delta > 0$. Therefore for sufficiently large $x$, we have 
		$$\frac{\#\{n\le x/p_0^t: n\in \mathcal{A}\}}{x} \ge c,$$
		where $c$ is a positive constant. This in turn implies that}
		$
		\#\{n\le x: n\in \mathcal{F}\}\gg x.\\	
		$
		
		\bigskip
		
		\begin{acknowledgements}
			The first author would like to thank Ariel Weiss and Shaul Zemel for several useful discussions and for sharing their ideas with him. The research of the first author was partially supported by the grant no. 692854 provided by the European Research Council (ERC).	
			The research of the second and the third authors was partially supported by the DST-SERB grants CRG/2020/004147 and ECR/2016/001359, respectively. We would like to thank the anonymous referees for their careful reading and some useful suggestions which improved the presentation of the article.
		\end{acknowledgements}


\begin{thebibliography}{20}
			\bibitem{and}
			A. Andrianov, {\em Euler products corresponding to Siegel modular forms of genus $2$},
			Russian Math. Surveys  {\bf 29} (1974), 45--116.
			 \bibitem{Apostol} T. M. Apostol, Introduction to analytic number theory, {\em Springer-Verlag} (1976).
			\bibitem{bre}
			S. Breulmann, {\em On Hecke eigenforms in the Maass space}, Math. Z. {\bf 232} (1999), no. 3, 
			527--530.
		%	\bibitem{chen}
		%	G. Chen and W. Li,
		%	{\em A note on arithmetic behavior of Hecke eigenvalues of Siegel cusp forms of degree two}, 	Int. J. Number Theory, DOI: 10.1142/S1793042122500026.
			\bibitem{dm}
			S. Das and W. Kohnen, {\em On sign changes of eigenvalues of Siegel cusp forms of genus 2 in prime powers}, Acta Arith. {\bf 183} (2018), no. 2, 167--172.
			\bibitem{deitmar}
			A. Deitmar, Automorphic forms, Translated from the 2010 German original, Universitext, {\em Springer}, London (2013).
			\bibitem{die02}
			L. V. Dieulefait, {\em On the images of the Galois representations attached to genus $2$ Siegel modular forms},
			J. Reine Angew. Math. {\bf 553} (2002), 183--200.
			\bibitem{erdos}
			P. Erd\"{o}s, {\em On the difference of consecutive terms of sequences defined by divisibility properties}, Acta Arith. {\bf 12} (1966), 175--182.
			\bibitem{gkp}
			S. Gun, W. Kohnen and B. Paul, {\em Arithmetic behaviour of Hecke eigenvalues of Siegel cusp forms of degree two},
			Ramanujan J. {\bf 54} (2021), no. 1, 43--62.
			\bibitem{gkp}
			S. Gun, W. Kohnen and P. Rath, {\em Simultaneous sign change of Fourier-coefficients of two cusp forms}, 	Arch. Math. (Basel) {\bf 105} (2015), no. 5, 413--424.
			\bibitem{kohnen}
			W. Kohnen, {\em Sign changes of Hecke eigenvalues of Siegel cusp forms of genus $2$}, Proc. Amer. Math. Soc. {\bf 135} (2007), 997--999.	
			\bibitem{klsw}
			E. Kowalski, Y.-K. Lau, K. Soundararajan and J. Wu, {\em On modular signs}, Math. Proc. Cambridge Philos. Soc. {\bf 149} (2010), 389--411.
			\bibitem{kkw}
			A. Kumar,  M. Kumari and A. Weiss, {\em Images of Galois representations for Siegel eigenforms and applications} (in preparation). 	
			\bibitem{kkw-lt}
			A. Kumar,  M. Kumari and A. Weiss, {\em On the Lang--Trotter conjecture for Siegel modular forms}, https://arxiv.org/abs/2201.09278.	
			\bibitem{kms}
			A. Kumar,  J. Meher and K. D. Shankhadhar, {\em Strong multiplicity one for Siegel cusp forms of degree two}, Forum Math. {\bf 33} (2021), no. 5, 1157--1167.
			%\bibitem[Ku81]{kur}
			%N. Kurokawa, {\em On Siegel eigenforms}, Proc. Japan Acad. Ser. A Math. Sci. {\bf 57} (1981), no. 1, 47--50.
			%\bibitem{lau}
			%G. Laumon, {\em Sur la cohomologie \`a supports compacts des vari\'et\'es de Shimura pour ${\rm GSp}(4)_{\mathbb Q}$}, Compositio Math. {\bf 105} (1997), no. 3, 267--359.
			\bibitem{loe}
			D. Loeffler, {\em Images of adelic Galois representations for modular forms},	Glasg. Math. J. {\bf 59} (2017), no. 1, 11--25.
			\bibitem{mat}
			K. Matom\"aki, {\em On signs of Fourier coefficients of cusp forms}, Math. Proc.
			Cambridge Philos. Soc. {\bf 152} (2012), 207--222.
			%\bibitem[Pi99]{pil}
			%R. Schulze-Pillot, {\em Siegel modular forms having the same $L$-functions}, J. Math. Sci. Univ. Tokyo {\bf 6 } (1999), no. 1, 217--227. 
			%	\bibitem{pss}
			%	A. Pitale, A. Saha and R. Schmidt, {\em Transfer of Siegel cusp forms of degree $2$}, Mem. Amer. Math. Soc. 232 (2014), no. 1090, vi+107 pp.\\
				\bibitem{prasad}
				G. Prasad and A. S. Rapinchuk, {\em Subnormal subgroups of the groups of rational points of reductive algebraic groups}, Proc. Amer. Math. Soc. {\bf 130} (2002), no. 8, 2219--2227.
			\bibitem{pit}
			A. Pitale, Siegel modular forms. A classical and representation-theoretic approach, 
			Lecture Notes in Mathematics, 2240. {\em Springer}, Cham, 2019. ix+138 pp. 
			\bibitem{pisc}
			A. Pitale and R. Schmidt, {\em  Sign changes of Hecke eigenvalues of Siegel cusp forms of
				degree 2}, Proc. Amer. Math. Soc. {\bf 136} (2008), no. 11, 3831--3838.
			\bibitem{raj}
			C. S. Rajan, {\em On strong multiplicity one for $l$-adic representations},
			Internat. Math. Res. Notices (1998), no. 3, 161--172.
		%	\bibitem{ram}
		%	D. Ramakrishnan, {\em Decomposition and parity of Galois representations attached to ${\rm GL}(4)$},
		%	Automorphic representations and L-functions, 427--454, 
		%	Tata Inst. Fundam. Res. Stud. Math., 22, Tata Inst. Fund. Res., Mumbai, 2013.
			\bibitem{rsw}
			E. Royer, J. Sengupta and J. Wu, {\em Non-vanishing and sign changes of Hecke eigenvalues for Siegel cusp forms of genus two} (with an appendix by E. Kowalski and A. Saha), Ramanujan J. {\bf 39} (2016), 179--199.
			\bibitem{sch}
			R. Schmidt, {\em Packet structure and paramodular forms},
			Trans. Amer. Math. Soc. {\bf 370} (2018), no. 5, 3085--3112.
			%\bibitem{tay}
			%R. Taylor, {\em Galois representations associated to Siegel modular forms of low weight}, Duke Math. J. {\bf 63} (1991), no. 2, 281--332. 
			%\bibitem{wei}
			%A. Weiss, {\em On Galois representations associated to low weight Hilbert-Siegel modular forms}, PhD thesis.
			\bibitem{wei}
			A. Weiss, {\em On the images of Galois representations attached to low weight Siegel modular forms}, J. Lond. Math. Soc. (to appear), DOI: 10.1112/jlms.12576.
			\bibitem{wei09}
			R. Weissauer, Endoscopy for ${\rm GSp}(4)$ and the cohomology of Siegel modular threefolds,
			volume 1968 of Lecture Notes in Mathematics, {\em Springer-Verlag}, Berlin, 2009.
		%	\bibitem{wei05}
		%	R. Weissauer, {\em Four dimensional Galois representations}, Ast\'erisque, No. 302 (2005), 67--150. 
		\end{thebibliography}
	\end{document}